\documentclass[reqno,11pt]{amsart}

\numberwithin{equation}{section}
\usepackage{amsmath}
\usepackage{esint} % permette di fare l'integrale tagliato
\usepackage{cancel}
\usepackage{amsthm}
\usepackage{epsfig}
\usepackage{psfrag}
\usepackage{graphicx}
\usepackage{graphpap,latexsym,epsf}
\usepackage{color}
\usepackage{amssymb,mathrsfs,enumerate}
\usepackage{mathtools}
\usepackage{subfigure}
\definecolor{citegreen}{rgb}{0,0.6,0}
\definecolor{refred}{rgb}{0.8,0,0}
\usepackage[colorlinks, citecolor=citegreen, linkcolor=refred]{hyperref}
\newcommand{\R}{\mathbb{R}}

\def\HHH{{\rm H}}
\def\RRR{{\mathrm R}}

\def\a{\alpha}
\def\b{\beta}

\newcommand{\pa}{\partial}

\newcommand{\ep}{\varepsilon}

\newcommand{\mcrit}{m_{{\rm crit}}}

\newcommand{\Ric}{{\rm Ric}}

\newcommand{\D}{{\nabla}}
\newcommand{\DD}{{\nabla^2}}
\newcommand{\De}{\Delta}

\newcommand{\na}{\nabla}

\newcommand{\hhh}{{\rm h}}

\mathchardef\emptyset="001F

\definecolor{vgreen}{rgb}{0.1,0.5,0.2}
\definecolor{viola}{RGB}{85,26,139}
\newtheorem{theorem}{Theorem}[section]
\newtheorem{remark}{Remark}

\newtheorem{definition}{Definition}
\newtheorem{proposition}[theorem]{Proposition}

\begin{document}

\hyphenation{ma-ni-fold ma-ni-folds de-ge-ne-ra-te asymp-to-tic}

\title
[Static Black Hole Uniqueness for nonpositive masses]{Static Black Hole Uniqueness for nonpositive masses}

\author[Stefano~Borghini]{Stefano Borghini}
\address{S.~Borghini, Uppsala Universitet, L\"{a}gerhyddsv\"{a}gen 1, 752 37 Uppsala, Sweden}
\email{stefano.borghini@math.uu.se}

%\author[P.T.~Chru\'sciel]{Piotr T. Chru\'sciel}
%\address{P.T.~Chru\'sciel, University of Vienna, Gravitational Physics
%Boltzmanngasse 5, A 1090 Vienna, Austria}
%\email{Piotr.Chrusciel@univie.ac.at}

%%\author[V.~Agostiniani]{Virginia Agostiniani}
%\address{V.~Agostiniani, SISSA, via Bonomea 265, 34136 Trieste, Italy}
%\email{vagostin@sissa.it}

%\author[L.~Mazzieri]{Lorenzo Mazzieri}
%\address{L.~Mazzieri, Universit\`a degli Studi di Trento, via Sommarive 14, 38123 Povo (TN), Italy}
%\email{lorenzo.mazzieri@unitn.it}

% \thanks{}

\begin{abstract}
In~\cite{Lee_Nev}, Lee and Neves proved a Penrose inequality for spacetimes with negative cosmological constant and nonpositive mass aspect. As an application,
% building on some previous results in~\cite{Chr_Sim}, 
they were able to obtain a static uniqueness theorem for Kottler spacetimes with nonpositive mass.
In this paper, we propose an alternative more elementary proof of this static uniqueness result and we discuss some possible generalizations, most notably to degenerate horizons.
\end{abstract}

\maketitle


\smallskip
\noindent 

\date{\today}

\maketitle

\section{Introduction}

%In this paper, we focus on {\em static vacuum spacetimes}. These are $4$-dimensional Lorentzian manifolds $(X,\gamma)$ satisfying the Einstein Field Equations in the vacuum $\Ric_\gamma\,=\,\Lambda\,\gamma$, where $\Lambda\in\R$ is the {\em cosmological constant}, and having the following warped product structure
%\begin{equation}
%\label{eq:spacetime}
%X\,=\,\R\,\times\,M\,,\qquad \gamma\,=\,-u^2\,dt\otimes dt\,+\,g\,.
%\end{equation}
%Here $(M,g)$ is a $3$-dimensional Riemannian manifold and $u:M\to\R$ is a smooth function, called {\em static potential}. 

%A {\em vacuum spacetime with cosmological constant $\Lambda$} is a $4$-dimensional Lorentzian manifold $(X,\gamma)$ satisfying $\Ric_\gamma\,=\,\Lambda\,\gamma$. 
%We are interested in the special case where our spacetime is {\em static}. Namely, we require that $(X,\gamma)$ has the following structure
%\begin{equation}
%\label{eq:spacetime}
%X\,=\,\R\,\times\,M\,,\qquad \gamma\,=\,-u^2\,dt\otimes dt\,+\,g\,,
%\end{equation}
%where $(M,g)$ is a $3$-dimensional Riemannian manifold and $u:M\to\R$ is a smooth function, called {\em static potential}. 

Static spacetimes represent the time independent solutions of the Einstein field equations, and as such they are some of the fundamental models in general relativity. It is then a natural problem to ask whether it is possible to classify them, starting from the vacuum case.
If the cosmological constant $\Lambda$ is equal to zero, static vacuum spacetimes are well understood: the celebrated Black Hole Uniqueness Theorem, proven in~\cite{Bun_Mas} (see also~\cite{Israel,Robinson,ZHa_Rob_Sei} for earlier contributions and~\cite{Ago_Maz_2} for a recent alternative approach) states that the only asymptotically flat static vacuum spacetimes are the Schwarzschild black hole and the Minkowski space. If one drops the assumption of asymptotic flatness, then other static vacuum spacetimes are known, namely the boosts and the Myers/Korotkin-Nicolai black holes. Characterizations of these solutions have been shown in the papers~\cite{Reiris_claI,Reiris_claII,Rei_Per}.

%It is also worth mentioning that some strong classification results have been proven even if one drops the assumption of asymptotic flatness, see~\cite{Reiris_claI,Reiris_claII,Rei_Per}. 

Such strong results are not available when $\Lambda\neq 0$. If the cosmological constant is positive, there are some characterizations of the known model solutions under some additional geometric hypotheses, see for instance~\cite{Ambrozio,Kobayashi,Lafontaine}.
%It is still unclear if these are the only ones, although partial interesting results have been proven in~\cite{Lafontaine,Kobayashi}, where they proved that there is no other solution under the locally conformally flat hypothesis, in~\cite{Bou_Gib_Hor,Ambrozio}, who obtained some geometric and topological result on the horizons
In~\cite{Bor_Maz_2-I,Bor_Maz_2-II} the notion of {\em pseudo-radial function} was introduced and it was shown how to effectively exploit it to perform a comparison with the model solutions, ultimately leading in~\cite{Bor_Chr_Maz} to a uniqueness result for the Schwarzschild--de Sitter spacetime.
Concerning the negative cosmological constant case, a characterization of the Anti de Sitter solution has been provided in~\cite{Bou_Gib_Hor} and it has then been improved in~\cite{Chr_Her,Qing,Wang_2}, whereas a static uniqueness theorem for the AdS Soliton has been shown in~\cite{Gal_Sur_Woo}. Finally, Lee and Neves in~\cite{Lee_Nev} proved a Penrose inequality for asymptotically hyperbolic manifolds and exploited it to prove a uniqueness result for Kottler solutions with nonpositive mass. 

In this paper we will show an alternative proof of this latter static uniqueness result. While the proof in~\cite{Lee_Nev} is based on the Penrose inequality, our argument is more elementary and it will rely on the notion of pseudo-radial function. As already mentioned, this function has been introduced in~\cite{Bor_Maz_2-II} for the purpose of studying static spacetimes with positive cosmological constant, but it can be easily adapted to the $\Lambda<0$ case, as we will discuss in more details in Section~\ref{sec:pr}.

Let us start by quickly recalling our setting. A {\em vacuum spacetime with cosmological constant $\Lambda$} is a $4$-dimensional Lorentzian manifold $(X,\gamma)$ satisfying the Einstein field equations $\Ric_\gamma\,=\,\Lambda\,\gamma$. Since we will only be concerned with the negative cosmological constant case, for the sake of simplicity we fix its value once and for all: up to a normalization of $\gamma$ we suppose $\Lambda=-3$.
%The spacetime $(X,\gamma)$ is {\em static} if it has the following structure
%\begin{equation}
%\label{eq:spacetime}
%X\,=\,\R\,\times\,M\,,\qquad \gamma\,=\,-u^2\,dt\otimes dt\,+\,g\,,
%\end{equation}
%where $(M,g)$ is a $3$-dimensional Riemannian manifold and $u:M\to\R$ is a smooth function, called {\em static potential}. 
We further assume that our spacetime is {\em static}, meaning that $(X,\gamma)$ has the following structure
\begin{equation}
\label{eq:spacetime}
X\,=\,\R\,\times\,M\,,\qquad \gamma\,=\,-u^2\,dt\otimes dt\,+\,g\,,
\end{equation}
where $(M,g)$ is a $3$-dimensional Riemannian manifold and $u:M\to\R$ is a smooth function, usually referred to as {\em lapse function} or {\em static potential}. It is also physically reasonable to ask for the function $u$ to be positive in the interior of $M$, with $u=0$ on $\pa M$, if nonempty. Finally, it is common to assume that $\pa M$ is compact and that the metric $g$ is smooth in $M$ and well defined on $\pa M$. In this framework, the Einstein field equations $\Ric_\gamma=-3\,\gamma$ translate into the following PDE problem for the static potential
\begin{equation}
\label{eq:pb}
\begin{dcases}
u\,\mbox{Ric}\,=\,\D^2 u\,-\,3 \,u\,g\,, & \mbox{in } M\\
\ \ \,\Delta u\,=\,3\,u\,, & \mbox{in } M\\
\ \ \ \ \ \,u=0, & \mbox{on } \partial M\\
\ \ \ \ \ \,u>0, & \mbox{in } \mathring{M}
\end{dcases}
%\quad\hbox{ with $\pa M$ compact,}
\end{equation}
where $\mbox{Ric},\D,\Delta$ are the Ricci tensor, the Levi-Civita connection and the Laplace-Beltrami operator of the metric $g$. For ease of reference, let us sum up the properties that we expect from $M$, $g$ and $u$ in the following definition.

\begin{definition}
\label{def:static_solution} 
A {\em static solution with cosmological constant $\Lambda=-3$} is a triple $(M,g,u)$ where $M$ is a connected smooth $3$-dimensional Riemannian manifold with compact boundary $\pa M$ (possibly empty), $g$ is a complete smooth Riemannian metric on $M$ and $u:M\to\R$ is a smooth function solving~\eqref{eq:pb}. Two static solutions $(M,g,u)$ and $(M',g',u')$ are said to be {\em isometric} if there is a diffeomorphism $\phi:M\to M'$ such that $\phi^* g'=g$ and $u'\circ\phi = c\, u$, for some constant $c>0$.  
\end{definition}

Some remarks are in order. First of all, taking the trace of the first equation in~\eqref{eq:pb} and using the second one, one immediately shows that the scalar curvature $\RRR$ of a static solution is constant, and more precisely that it holds
$$
\RRR\,\,=\,\,-\,6\,.
$$ 
Furthermore, since $u=0$ on $\pa M$, again from the first equation we deduce that the hessian $\DD u$ vanishes pointwise on $\pa M$. Starting from this, standard arguments (see for instance~\cite{Ambrozio,Hij_Mon}) allow to show that the components of $\pa M$ are totally geodesic and $|\D u|$ is constant and positive on each of them.  A connected component $S$ of $\pa M$ is usually called {\em horizon} and the constant value of $|\D u|$ on $S$ is referred to as the {\em surface gravity} of $S$. Finally, we observe that $M$ cannot be compact, otherwise, since
$\De u=3u\geq 0$ in $M$, the Strong Maximum Principle would tell us that $u$ attains its maximum on the boundary of $M$, in clear contrast with~\eqref{eq:pb}. 
It follows that $M$ must have at least one end. We will need the following standard hypothesis on the behavior of the triple $(M,g,u)$ near infinity:

\begin{definition}
Let $i\geq 2$ be an integer. A static solution $(M,g,u)$ is said to be {\em $\mathscr{C}^i$-compactifiable} if the following properties hold:
\smallskip
\begin{itemize}
\item there is a $\mathscr{C}^{i+1}$-diffeomorphism between $M\setminus\pa M$ and the interior of a smooth compact manifold $\overline{M}$ with $\pa\overline{M}=\pa M\sqcup \pa M_\infty$,

\smallskip

\item  the function $u^{-1}$ extends to a $\mathscr{C}^i$ function in a collar neighborhood of $\pa M_\infty$ in $\overline{M}$ in such a way that $(u^{-1})_{|_{\pa M_\infty}}=0$ and $d(u^{-1})_{|_{\pa M_\infty}}\neq 0$,
\smallskip
\item the metric $u^{-2}g$ extends to a $\mathscr{C}^i$ metric in a collar neighborhood of $\pa M_\infty$ in $\overline{M}$.
\end{itemize} 
We will denote by $\hat g$ the metric induced by $u^{-2}g$ on $\pa M_\infty$. The pair $(\pa M_\infty,\hat g)$ is called the {\em conformal infinity} of $(M,g,u)$.
\end{definition}

It is known from~\cite[Theorem~I.1]{Chr_Sim} that the conformal infinity $\pa M_\infty$ of a $\mathscr{C}^3$-conformally compactifiable static solution is always {\em connected}. We will be mostly concerned with the case in which the metric $\hat g$ has constant sectional curvature $\kappa_{\hat g}$. Under this additional hypothesis, up to a normalization of $u$, it is not restrictive to assume that $\kappa_{\hat g}=+1$, $0$ or $-1$, depending on whether the genus of the surface $\pa_{\infty} M$ is equal to $0$, equal to $1$ or greater than $1$, respectively.

\begin{figure}
\centering
\begin{subfigure}%[b]%{0.5\textwidth}
\centering
\includegraphics[scale=0.28]{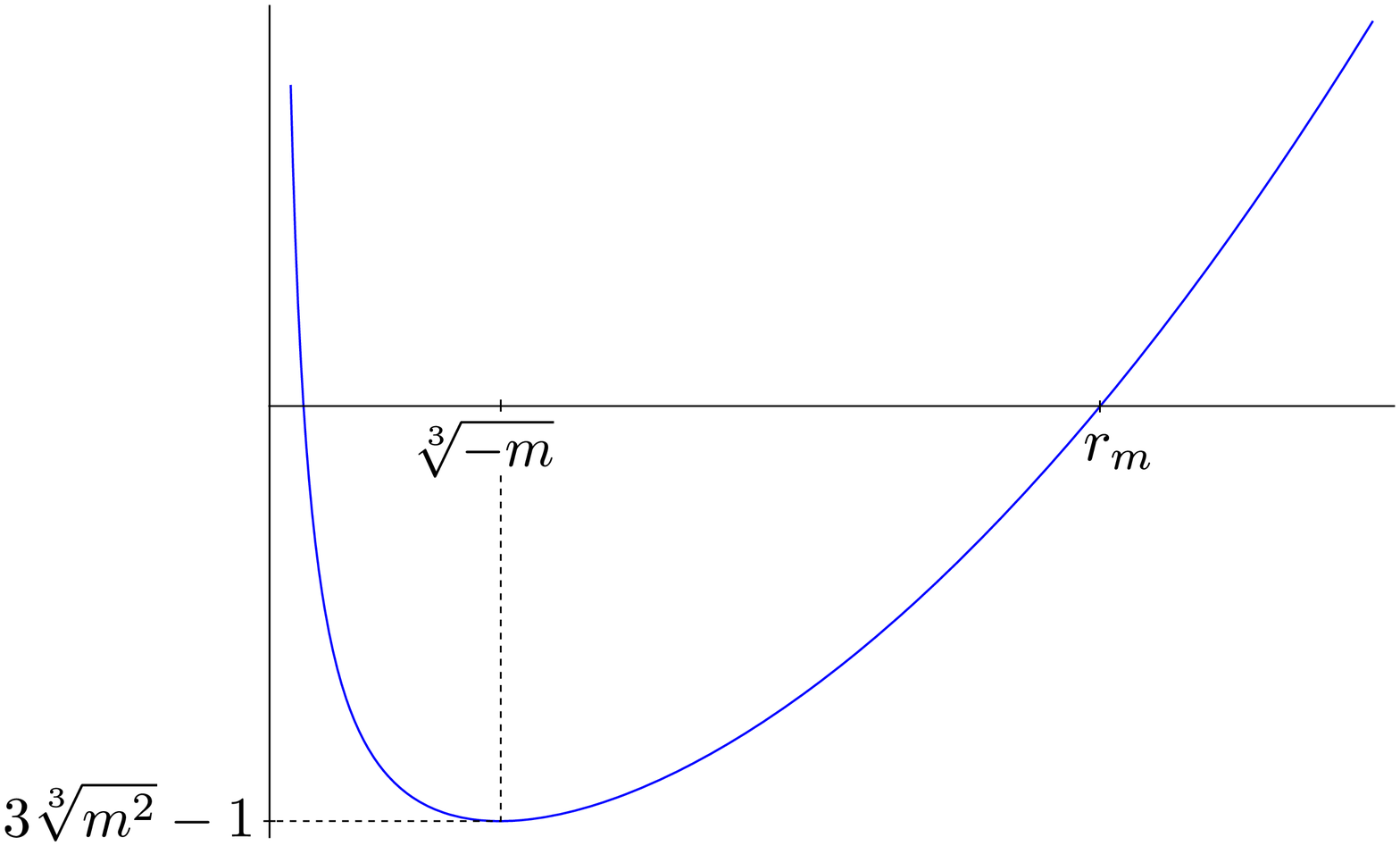}
%\caption{Lorem ipsum}
\end{subfigure}%
%\quad
%\begin{subfigure}%[b]%{0.5\textwidth}
%\centering
%\includegraphics[scale=0.23]{figure0}
%%\caption{}
%\end{subfigure}
\qquad\quad
\begin{subfigure}%[b]%{0.5\textwidth}
\centering
\includegraphics[scale=0.28]{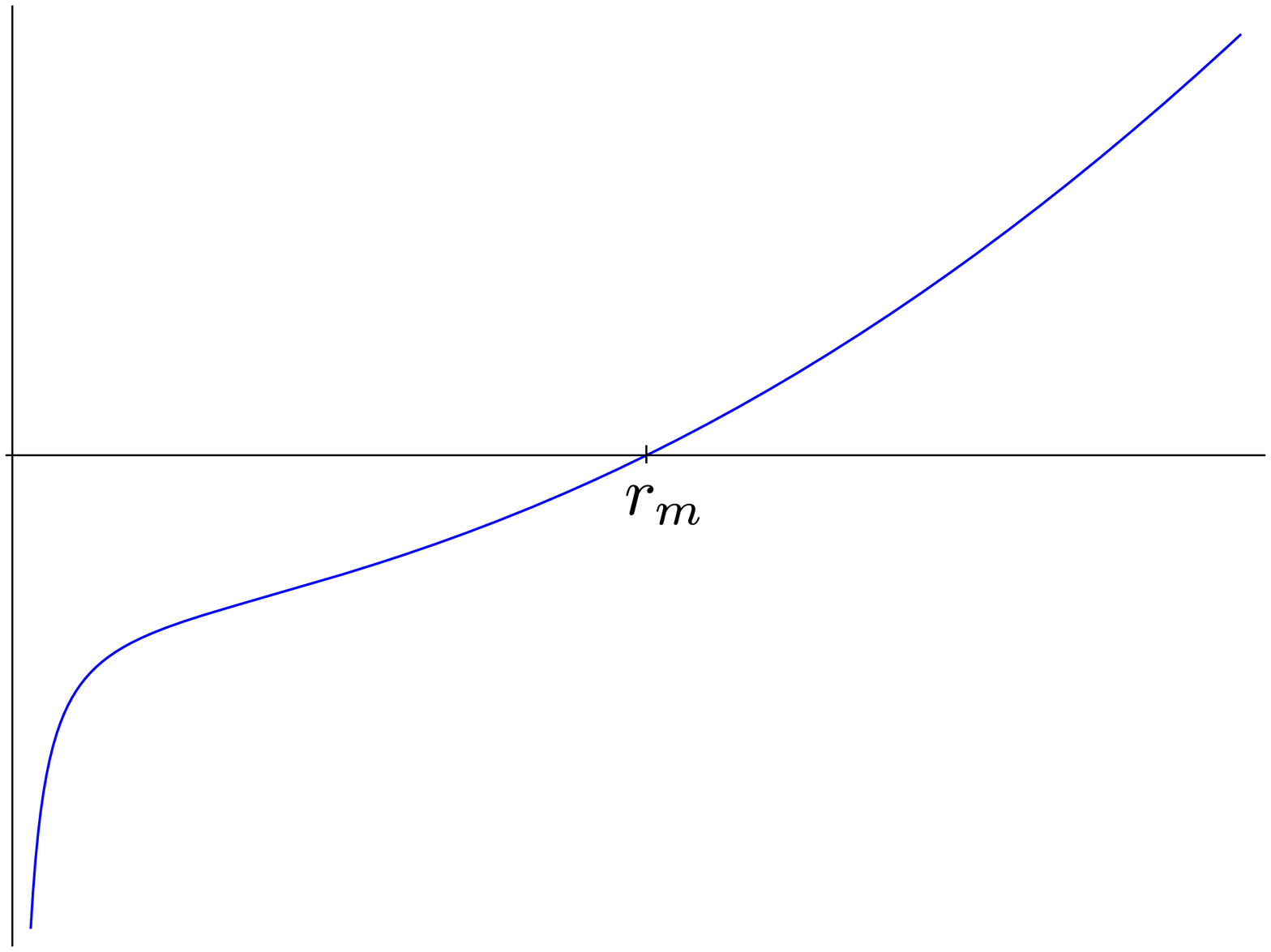}
%\caption{}
\end{subfigure}
\caption{\small Plot of $-1+r^2-2m/r$ as a function of $r$ for $m<0$ (left) and for $m>0$ (right). Notice that, for negative values of $m$, in order to have a positive zero $r_m$ of the function, one needs $3\sqrt[3]{m^2}-1<0$, which corresponds to $m>-1/(3\sqrt{3})$.}
\label{fig:kottler}
\end{figure}

Among the known solutions to~\eqref{eq:pb}, we are interested in the following family, usually referred to as the {\em Kottler solutions}~\cite{Kottler}
\begin{equation}
\label{eq:kottler}
M\,=\,[r_m,+\infty)\times\Sigma_\kappa\,,\quad
g\,=\,\frac{dr\otimes dr}{\kappa\,+\,\,r^2\,-\,\frac{2m}{r}}+r^2 g_{\Sigma_\kappa}\,,\quad
u\,=\,\sqrt{\kappa\,+\,r^2\,-\,\frac{2m}{r}}\,,
\end{equation}
where $(\Sigma_\kappa,g_{\Sigma_\kappa})$ is a surface with constant sectional curvature $\kappa$ and $r_m>0$ is the positive root of the polynomial $\kappa x+x^3-2m=0$. In order for such a positive solution $r_m$ to exist, one needs to assume that the {\em mass} $m$ is in the appropriate range.
If $\kappa\geq 0$, it is easily seen that, in order for $r_m$ to exist, it is necessary and sufficient to take $m>0$. Let us now focus on the case $\kappa<0$. 
First of all, 
%for the sake of simplicity, 
%we notice that, 
up to a rescaling of the coordinate $r$, we can set $\kappa=-1$.
One can then easily study the plot of the function $u^2=-1+r^2-2m/r$ (see Figure~\ref{fig:kottler}) to show that $r_m$ exists if and only if
$$
m\,\,>\,\,-\,\frac{1}{3\,\sqrt{3}}\,\,.
$$
Notice in particular that, when $\kappa<0$, negative masses are allowed. Let us mention that one can make sense of the Kottler solution with mass equal to the critical value $\mcrit = -1/(3\sqrt{3})$, in which case the horizon $\pa M$ becomes {\em degenerate}. We will postpone the discussion of this special case to Section~\ref{sec:future}.
Finally, for future reference, we write down the formula for the norm of the gradient of the static potential of a Kottler solution:
\begin{equation}
\label{eq:grad_kottler}
|\D u|^2\,=\,\left(\frac{r^3+m}{r^2}\right)^2\,.
\end{equation}
As we will see in the following sections, this quantity plays a central role in the analysis.

%\begin{figure}
%\begin{center}
%\includegraphics[scale=0.4]{figure1}
%\caption{Plot of the function $-1+r^2-2m/r$. Notice that, for negative values of $m$, for the zero $r_m$ to exist, one needs $3\sqrt[3]{m^2}-1<0$, which corresponds to $m>-1/(3\sqrt{3})$.}
%\label{fig:kottler}
%\end{center}
%
%\end{figure}

We are now ready to state the uniqueness result for Kottler solutions with negative mass that we want to prove:

\begin{theorem}[Static Uniqueness for nonpositive masses]
\label{thm:main}
Let $(M,g,u)$ be a $\mathscr{C}^3$-conformally compactifiable static solution with cosmological constant $\Lambda=-3$ and suppose that the conformal infinity $(\pa M_\infty,\hat g)$ has constant sectional curvature equal to $-1$. Let $S$ be an horizon with maximum surface gravity $k$, and suppose that $0<k\leq 1$. If 
$$
{\rm genus}(\pa M_\infty)\,\leq\,{\rm genus}(S)\,,
$$
then $(M,g,u)$ is isometric to a Kottler solution~\eqref{eq:kottler} with nonpositive mass.
\end{theorem}

\begin{remark}
\label{rem:main}
Notice that this result is slightly stronger than the one proved by Lee and Neves in~\cite[Theorem~2.1]{Lee_Nev}, as we only need $\mathscr{C}^3$-compactifiability instead of $\mathscr{C}^5$. It is also worth mentioning that, as we will discuss in more details in Section~\ref{sec:future}, our proof can be adapted to treat the degenerate case $k=0$ as well, leading to a characterization of the critical Kottler solution.
\end{remark}

The paper is structured as follows. In Section~\ref{sec:LN} we will recall the main steps in the proof of Lee and Neves. This will give us the chance to introduce a couple of important preliminary results, namely Proposition~\ref{pro:grad_est} and Proposition~\ref{pro:area_bound}. These two propositions, proven by Chru\'sciel and Simon in~\cite{Chr_Sim}, will play an important role in our proof.
In Section~\ref{sec:pr} we will define the pseudo-radial function $\Psi$, which will then be exploited in the computations in Section~\ref{sec:proof}, leading to the proof of Theorem~\ref{thm:main}. Section~\ref{sec:future} is devoted to the discussion of possible extensions and generalizations. In Subsection~\ref{sub:degenerate} we will discuss how to adapt our proof to include the case of degenerate horizons, whereas in Subsection~\ref{sub:relax_hyp} we will comment on the hypotheses of Theorem~\ref{thm:main}, showing that we can relax some of them and still obtain some partial results. Finally, in Subsection~\ref{sub:generalizations} we briefly discuss some further generalizations and possible future directions.

\section{The proof of Lee and Neves}
\label{sec:LN}

%We start by recalling the main steps in the proof of the static uniqueness by Lee and Neves~\cite{Lee_Nev}. 
In this section we recall the main steps in the proof of the static uniqueness for Kottler solutions provided in~\cite{Lee_Nev} by Lee and Neves.
Let $(M,g,u)$ be a static solution with cosmological constant $\Lambda=-3$ and suppose that the maximum $k>0$ of the surface gravities of its horizons is less than or equal to $1$. Then one can easily check that there is exactly one value $m_0\leq 0$ such that the horizon of the Kottler solution $(M_0,g_0,u_0)$ with mass $m_0$ has surface gravity equal to $k$. We want to compare our general solution $(M,g,u)$ with the Kottler solution $(M_0,g_0,u_0)$.
Two functions that will be crucial for said comparison are $W=|\D u|^2$ and $W_0:M\to \R$, representing the corresponding quantity on the reference model $(M_0,g_0,u_0)$. Namely, given $p\in M$, the number $W_0(p)$ is equal to the value of $|\D u_0|_{g_0}^2$ on the level set $\{u_0=u(p)\}\subseteq M_0$. Notice that the function $W_0$ is well defined, since $|\D u_0|_{g_0}^2$ is constant on each level set of $u_0$. Chru\'sciel and Simon in~\cite{Chr_Sim}, following a previous calculation in~\cite{Bei_Sim}, show that the quantity $W-W_0$ satisfies an elliptic inequality of the form
\begin{equation}
\label{eq:elliptic}
\De(W-W_0)\,+\,\langle \xi\,|\,\D(W-W_0)\rangle\,+\,\alpha(W-W_0)\,\geq\,0
\end{equation}
on any domain $\Omega\subseteq M$ not containing any critical points of $u$,
where $\xi$ is smooth and $\alpha$ has the same sign as $m_0$.  See also~\cite[Lemma~4.1]{Lee_Nev} for more details on the computations.
%{\color{blue}There are some technicalities that we are not addressing: for more details, see~\cite{Chr_Sim} and also~\cite[Section~4]{Lee_Nev}. }
Notice in particular that, since we are assuming $m_0\leq 0$, the function $\alpha$ is nonpositive.
Furthermore $W_0$ has been defined in such a way that $W=W_0$ on the horizon with maximum surface gravity and $W\leq W_0$ on eventual other horizons, hence $W\leq W_0$ on $\pa M$. Finally, under the assumption that $(M,g,u)$ is conformally compact and the conformal infinity has constant curvature equal to $-1$, Chru\'sciel and Simon computed the asymptotic behavior of $W-W_0$, showing in particular that $W-W_0\to 0$ at infinity. As a consequence of the Maximum Principle, one obtains the following
\begin{proposition}
\label{pro:grad_est}
Let $W$, $W_0$ be defined as above. In the hypotheses of Theorem~\ref{thm:main}, on the whole $M$ it holds
\begin{equation}
\label{eq:W_leq_Wo}
W\,\,\leq\,\,W_0\,.
\end{equation}
\end{proposition}

The details can be found in~\cite[Subsection~VII.B]{Chr_Sim} or~\cite[Corollary~4.2]{Lee_Nev}. 
%Since this part of the argument is crucial for our method as well, in Section~\ref{} we will retrace it again, to show that~\eqref{eq:W_leq_Wo} remains true under weaker assumptions on the asymptotics. 
Inequality~\eqref{eq:W_leq_Wo} has some interesting consequences, discussed in~\cite[Subsection~VII.B]{Chr_Sim} and~\cite[Section~4]{Lee_Nev}. An expansion of $W$ and $W_0$ near an horizon $S$ with maximum surface gravity yields a bound on the area of $S$ in terms of the area of the model solutions:
$$
\frac{\chi(S)}{|S|}\,\geq\,\frac{\chi(\pa M_0)}{|\pa M_0|_{g_0}}\,.
$$
Let us observe that we can compute $|\pa M_0|_{g_0}$ explicitly. In fact, we observe from~\eqref{eq:kottler} that $\pa M_0$ is endowed with the metric $r^2_{m_0}\, g_{-1}$, where $g_{-1}$ is the metric with constant sectional curvature $-1$.
As a consequence of the Gauss Bonnet Theorem, we obtain then:
$$
|\pa M_0|\,=\,r_{m_0}^2\int_{\pa M_0}d\sigma_{g_{-1}}\,=\,-\,r_{m_0}^2\,\int_{\pa M_0}\!\kappa_{g_{-1}}\,d\sigma_{g_{-1}}\,=\,-2\,\pi\,r_{m_0}^2\,\chi(\pa M_0)\,.
$$
Since $\chi(S)=-2\,[{\rm genus}(S)-1]$, the next result follows immediately:

\begin{proposition}
\label{pro:area_bound}
In the hypotheses of Theorem~\ref{thm:main}, it holds
$$
|S|\,\geq\,4\,\pi\,r_{m_0}^2\,\big[{\rm genus}(S)-1\big]\,.
$$
\end{proposition}
If instead one looks at the asymptotic expansions of $W$ and $W_0$, assuming again that $(M,g,u)$ is conformally compact and that the conformal infinity has constant curvature equal to $-1$,
%(here this assumption is crucial, so this part of the argument does not work in our framework) 
then one shows that $\overline{m}\leq m_0$, where $\overline{m}$ is the supremum of the mass aspect. 
In~\cite[Subsection~2.1]{Lee_Nev}, Lee and Neves are able to combine the area bound on the horizon $S$, the bound $\overline{m}\leq m_0$ and their Penrose Inequality~\cite[Theorem~1.1]{Lee_Nev} to create a chain of inequalities, the former and latter terms of which are equal. Their uniqueness result then follows easily.

The arguments in the latter paragraph, involving the mass aspect, will not be relevant for what concerns our proof. Conversely, as we will see in the next sections, Propositions~\ref{pro:grad_est} and~\ref{pro:area_bound} will play a crucial role.

\section{Pseudo-radial function}
\label{sec:pr}
%
%In Section~\ref{sec:LN} we have introduced the quantity $m_0$, representing the mass of the model solution wit
%
%The quantity $m_0$ introduced in the previous section is the analogue of the notion of {\em virtual mass} discussed in~\cite{Bor_Maz_2-I}
%
%
%
%In the papers~\cite{Bor_Maz_2-I, Bor_Maz_2-II}, the notions of virtual mass and pseudo-radial functions were introduced to study static spacetimes with positive cosmological constant. The virtual mass is essentially the analogue of the parameter $m_0$ that we have introduced in the previous section. 
%
%Building on it, the {\em pseudo-radial function} was defined in~\cite{Bor_Maz_2-II}.

With {\em pseudo-radial function} we mean a function that mimics the behavior of the radial coordinate of a model solution.
In~\cite[Subsection~2.1]{Bor_Maz_2-II} such notion was introduced to study static spacetimes with positive cosmological constant, using the Schwarzschild--de Sitter spacetime as a model. 
Here we show that, with only minor modifications, we can define an analogous notion in the negative cosmological constant setting as well, using the Kottler solution~\eqref{eq:kottler} as a model.

%{This section is devoted to the definition of the {\em pseudo-radial function}, which is a function that mimics the behavior of the radial coordinate of the Kottler solution. Such definition has been already introduced in~\cite[Subsection~2.1]{Bor_Maz_2-II} in the positive cosmological constant framework. Here we show that such a definition can be exported to the negative cosmological constant setting with minor modifications. }

%For the sake of completeness, let us give the precise definition. 
Fix $m\in(-1/3\sqrt{3},0]$ and consider the function
$$
\begin{aligned}
&F_m :\,[0,+\infty)\,\times\,[r_m,+\infty)\,\longrightarrow\,\R
\\
&\,(u,\psi)\,\longmapsto\,u^2\,+\,1\,-\,\psi^2\,+\,\frac{2\,m}{\psi}\,.
\end{aligned}
$$
Notice that $\pa F_m/\pa\psi=-2(\psi^3+m)/\psi^2$ is strictly negative for all $\psi>\sqrt[3]{-m}$. In particular, since $r_m>\sqrt[3]{-m}$ (see Figure~\ref{fig:kottler}), we have $\pa F_m/\pa\psi\neq 0$ for all $m\in(-1/3\sqrt{3},0]$.
It follows from the Inverse Function Theorem that it is well defined a function $\psi_m:[0,\infty)\to [r_m,+\infty)$ such that $F_m(u,\psi_m(u))=0$, for all $u\in[0,+\infty)$. 
%In particular, it follows from the definition of $F_m$ that the function $\psi$ satisfies the following identity
%\begin{equation}
%u^2\,=\,-1+\psi^2(u)-\frac{2m}{\psi(u)}\,.
%\end{equation}
We now want to apply the function $\psi_m$ thus defined to the static potential. This is the content of the next definition.

\begin{definition}
Let $(M,g,u)$ be a solution to problem~\eqref{eq:pb} and let $m_0\leq 0$ be defined as in Section~\ref{sec:LN}. The {\em pseudo-radial function} is defined as
\begin{equation}
\label{eq:pr}
\begin{split}
\Psi:\,M\,\longrightarrow\,[r_{m_0},+\infty)
\\
p\,\longmapsto\,\psi_{m_0}\circ u(p)\,.\ \ \ 
\end{split}
\end{equation}
\end{definition}

From the definition of $\Psi$ and the fact that $F_{m_0}(u,\psi_{m_0}(u))=0$, we immediately obtain the following relation between $u$ and $\Psi$:
\begin{equation}
\label{eq:u_pr}
u^2\,=\,-\,1\,+\,\Psi^2\,-\,\frac{2\,m_0}{\Psi}\,.
\end{equation}
In particular, if $(M,g,u)$ is a Kottler solution~\eqref{eq:kottler} with $\kappa=-1$, then $\Psi$ coincides with the 
radial coordinate $r$. 
As a consequence, recalling~\eqref{eq:grad_kottler}, it is clear that we can write $W_0$ in terms of $\Psi$ as follows:
\begin{equation}
\label{eq:W0}
W_0\,=\,\left(\frac{\Psi^3+m_0}{\Psi^2}\right)^2\,.
\end{equation}
Finally, taking the derivative of both sides of identity~\eqref{eq:u_pr}, we obtain the following relation between the gradient of $\Psi$ and the gradient of $u$: 
\begin{equation}
\label{eq:grad_psi}
\D\Psi\,=\,\frac{u\,\Psi^2}{\Psi^3+m_0}\,\D u\,.
\end{equation}
%In particular, we can rewrite Proposition~\ref{pro:grad_est} in terms of $\Psi$ as follows
%\begin{proposition}
%\label{pro:grad_est_psi}
%Let $(M,g,u)$ be a $\mathscr{C}^3$-conformally compactifiable solution to problem~\eqref{eq:pb} and let $\Psi$ be the pseudo-radial function, defined as in~\eqref{eq:pr}. Then on the whole $(M,g,u)$ it holds
%$$
%|\D u|\,\leq\,\frac{\Psi^3+m}{\Psi^2}\,.
%$$
%\end{proposition}

\section{Proof of Theorem~\ref{thm:main}}
\label{sec:proof}
 
Let $(M,g,u)$ be a static solution with cosmological constant $\Lambda=-3$. 
As a consequence of the Divergence Theorem, on any compact domain $\Omega\subset M$ we have
\begin{align*}
\int_{\pa\Omega}\bigg\langle \frac{\D u}{\Psi^3+m_0}\,\bigg|\,\nu\bigg\rangle\,d\sigma
\,&=\,
\int_\Omega {\rm div}\left(\frac{\D u}{\Psi^3+m_0}\right)\,d\mu
\\
&=\,
\int_\Omega\left(\frac{\De u}{\Psi^3+m_0}\,-\,\frac{3\,\Psi^2}{(\Psi^3+m_0)^2}\,\langle\D\Psi\,|\,\D u\rangle\right)d\mu
\\
&=\,
\int_\Omega \frac{3\,u\,\Psi^4}{(\Psi^3+m_0)^3}\left[\frac{(\Psi^3+m_0)^2}{\Psi^4}\,-\,|\D u|^2\right]d\mu\,,
\end{align*}
where $\nu$ is the outward unit normal to $\pa \Omega$, and in the latter equality we have used $\De u=3\,u$ and~\eqref{eq:grad_psi}. Recalling the expression~\eqref{eq:W0} of $W_0$ in terms of $\Psi$, we observe that the quantity that appears in square brackets is exactly $W_0-W$. Therefore, for any choice of the domain $\Omega$, we have obtained
\begin{equation}
\label{eq:nonnegative_div}
\int_{\pa\Omega}\bigg\langle \frac{\D u}{\Psi^3+m_0}\,\bigg|\,\nu\bigg\rangle\,d\sigma
\,=\,
\int_\Omega \frac{3\,u\,\Psi^4}{(\Psi^3+m_0)^3}\left(W_0-W\right)d\mu\,\geq 0\,.
\end{equation}

This inequality will be crucial in the proof of the next proposition. Before stating it, let us recall some classical results on the behavior near the conformal infinity $(\pa M_\infty,\hat g)$ of a conformally compactifiable static solution $(M,g,u)$. 
It is well known that there exists a diffeomorphism between a collar of the conformal infinity and $[1,+\infty)\times \pa M_\infty$ such that the metric and the static potential can be written as
\begin{equation}
\label{eq:strong_expansions}
u\,=\,\rho\,+\,\frac{\kappa_{\hat g}}{2}\,\frac{1}{\rho}\,+\,\omega\,,\,\qquad g\,=\,\frac{d\rho\otimes d\rho}{\rho^2\,+\,\kappa_{\hat g}}\,+\,\rho^2\,\hat g\,+\,\eta\,,
\end{equation}
where $\rho$ is the coordinate on $[1,+\infty)$ and $\omega=o_1(\rho^{-1})$, $\eta_{ij}=o(1)$ as $\rho\to\infty$. With the notation $o_1$ we mean that the first derivative of $\omega$ scales suitably, namely, we want $\na\omega=o(\rho^{-2})$.
We also mention that it is actually possible to write much more refined expansions (see for instance the ones in~\cite[Proposition~2.2]{Lee_Nev} and~\cite[Proposition~III.7]{Chr_Sim}), however the ones in~\eqref{eq:strong_expansions} will be enough for our purposes. 

\begin{proposition}
\label{pro:mono}
In the hypotheses of Theorem~\ref{thm:main}, we have
\begin{equation}
\label{eq:mono}
|\pa M_\infty|_{\hat g}\,\geq\,4\,\pi\,\big[{\rm genus}(S)-1\big]\,,
\end{equation}
where $S$ is the horizon with maximum surface gravity. Furthermore, if the equality holds, then $W\equiv W_0$ on the whole $M$.
\end{proposition}

\begin{proof}
Let $\rho$ be the coordinate introduced above in a collar of infinity and let us apply formula~\eqref{eq:nonnegative_div} to the domain $\Omega=M\setminus\{\rho>R\}$, for some large $R>0$. Notice that
$$
\pa\Omega\,=\,\pa M\,\sqcup\,\{\rho=R\}\,,
$$
and that, on $\pa M$, the unit normal $\nu$ is equal to $-\D u/|\D u|$.
Let now $m_0\leq 0$ be the mass of the Kottler solution we are comparing with, and let $\Psi$ be the pseudo-radial function defined as in~\eqref{eq:pr}.
From the definition, it is clear that $\Psi=r_{m_0}$ on $\pa M$, hence, recalling~\eqref{eq:W0}, on $\pa M$ we have $\sqrt{W_0}=(r_{m_0}^3+m_0)/r_{m_0}^2$. It follows
$$
\bigg\langle \frac{\D u}{\Psi^3+m_0}\,\bigg|\,\nu\bigg\rangle\,=\,-\,\frac{|\D u|}{r_{m_0}^3+m_0}\,=\,-\,\sqrt{\frac{W}{W_0}}\,r_{m_0}^{-2}\,.
$$
Recalling that $W=W_0$ on the horizon $S$ with maximum surface gravity, formula~\eqref{eq:nonnegative_div} gives us
$$
\int_{\{\rho=R\}}\bigg\langle \frac{\D u}{\Psi^3+m_0}\,\bigg|\,\nu\bigg\rangle\,d\sigma\,\geq\,
r_{m_0}^{-2}\int_{\pa M}\sqrt{\frac{W}{W_0}}\,d\mu
\,\geq\,
r_{m_0}^{-2}\,|S|
\,,
$$
Furthermore, it is clear from~\eqref{eq:nonnegative_div} that, if the former inequality is an equality, then $W\equiv W_0$ on the whole $\Omega$.
Taking the limit of the left hand side as $R\to \infty$ and using Proposition~\ref{pro:area_bound}, we have proven
\begin{equation}
\label{eq:int_id_aux}
\lim_{R\to \infty}\int_{\{\rho=R\}}\bigg\langle \frac{\D u}{\Psi^3+m_0}\,\bigg|\,\nu\bigg\rangle\,d\sigma\,\geq\,
4\,\pi\,\big[{\rm genus}(S)-1\big]
\,.
\end{equation}

It remains to compute the limit on the left hand side of~\eqref{eq:int_id_aux}. 
Remembering~\eqref{eq:u_pr}, we have that $\Psi$ approaches $\rho$ at infinity, hence on $\{\rho=R\}$ it holds
$$
\frac{1}{\Psi^3+m_0}\,=\,\frac{1}{R^3}+o(R^{-3})\,.
$$
Furthermore, from~\eqref{eq:strong_expansions} and the fact that $\kappa_{\hat g}=-1$ by hypothesis, we have $|\D \rho|^2=\rho^2-1$ in the whole collar of infinity, hence the normal to $\{\rho=R\}$ satisfies
$$
\nu\,=\,\frac{\D \rho}{|\D \rho|}\,=\,\frac{1}{\sqrt{R^2-1}}\,\frac{\pa}{\pa\rho}\,.
$$ 
As a consequence
$$
\langle\D u\,|\,\nu\rangle\,=\,{\sqrt{R^2-1}}\,\frac{\pa u}{\pa\rho}\,=\,{\sqrt{R^2-1}}\,\left[1+\frac{1}{2\,R^2}+o(R^{-2})\right]\frac{\pa}{\pa\rho}\,=\,\left[R\,+\,o(R)\right]\frac{\pa}{\pa\rho}\,.
$$
Finally, from~\eqref{eq:strong_expansions} we also deduce
$$
d\sigma\,=\,\left[R^2+o(R^2)\right]\,d\sigma_{\hat g}\,.
$$
Putting all these pieces of information together, we get
\begin{equation}
\label{eq:limit_div}
\lim_{R\to \infty}\int_{\{\rho=R\}}\bigg\langle \frac{\D u}{\Psi^3+m_0}\,\bigg|\,\nu\bigg\rangle\,d\sigma\,=\,
\lim_{R\to\infty}\int_{\{\rho=R\}}d\sigma_{\hat g}\,=\,
\int_{\pa M_\infty}\,d\sigma_{\hat g}\,=\,|\pa M_\infty|_{\hat g}\,.
\end{equation}
This proves~\eqref{eq:mono}. Furthermore, notice that, if the equality holds, then formula~\eqref{eq:nonnegative_div} tells us that $W-W_0\equiv 0$ on $\Omega_\ep$ for all $\ep$, hence $W\equiv W_0$ on the whole $M$.
\end{proof}

We are finally ready to prove the main result of the paper, namely Theorem~\ref{thm:main}, that we rewrite here for the reader's convenience.

\begin{theorem}
\label{thm:main_rewr}
Let $(M,g,u)$ be a $\mathscr{C}^3$-conformally compactifiable static solution with cosmological constant $\Lambda=-3$ and suppose that the conformal infinity $(\pa M_\infty,\hat g)$ has constant sectional curvature equal to $-1$. Let $S$ be an horizon with maximum surface gravity $k$, and suppose that $0<k\leq 1$. If 
$$
{\rm genus}(\pa M_\infty)\,\leq\,{\rm genus}(S)\,,
$$
then $(M,g,u)$ is isometric to a Kottler solution~\eqref{eq:kottler} with negative mass.
\end{theorem}

\begin{proof}%[Proof of Theorem~\ref{thm:main}]
Under the hypothesis that the sectional curvature $\kappa_{\hat g}$ is constant and equal to $-1$ on $\pa M_\infty$, we can use the Gauss-Bonnet Theorem to write 
\begin{equation}
\label{eq:lim_theo}
|\pa M_\infty|_{\hat g}\,=\,-\int_{\pa M_\infty}\!\!\kappa_{\hat g}\,d\sigma_{\hat g}\,=\,-2\pi\chi(\pa M_\infty)\,=\,4\pi\left[{\rm genus}(\pa M_\infty)-1\right]\,\leq\,4\pi\left[{\rm genus}(S)-1\right]\,.
\end{equation} 
As a consequence, the equality holds in~\eqref{eq:mono}, so the rigidity statement of Proposition~\ref{pro:mono} applies, telling us that  
$$
W\equiv W_0 \quad\hbox{in $M$\,.}
$$
In particular, $|\D u|^2=W$ is a function of $u$, and it is strictly positive because $W_0\neq 0$ on the whole $M$ by definition. Since $u$ has no critical points, we can use it as a coordinate. Considering local coordinates $\{u,\vartheta^1,\vartheta^2\}$, the metric $g$ rewrites as
\begin{equation}
\label{eq:almost_warp_prod}
g\,=\,\frac{d u\otimes du}{W}\,+\,g_{ij}(u,\vartheta^1,\vartheta^2)d\vartheta^i\otimes d\vartheta^j\,,
\end{equation}
where the indices $i,j$ take only the values $1$ and $2$.
With respect to the normal $\D u/|\D u|$, the second fundamental form $\hhh$ and the mean curvature $\HHH$ of a level sets of $u$ can be computed by 
$$
\hhh_{ij}\,=\,\frac{\D^2_{ij} u}{|\D u|}\,,\qquad \HHH\,=\,\frac{\De u\,-\,\D^2_{uu} u}{|\D u|}\,.
$$
Let us now study the hessian of $u$. First of all, we directly compute
$$
\D^2_{\a\b} u\,=\,\pa^2_{\a\b}u\,-\,\Gamma_{\a\b}^\gamma\pa_\gamma u\,=\,-\Gamma_{\a\b}^u\,=\,\begin{dcases}
\frac{W'}{2 W} & \hbox{if $\a=\b=u$}\,,
\\
-\frac{W}{2}\pa_u g_{\a\b}              & \hbox{if both $\a\neq u$ and $\b\neq u$}\,,
\\             
0              & \hbox{otherwise}\,,
\end{dcases}
$$
where $W'$ represents the derivative of $W$ with respect to $u$. More precisely, $W':M\to\R$ is the function such that $\D W=W'\D u$.
%Let us denote by $h$ the metric $h_{ij}dx^i\otimes dx^j$ induced  by $g$ on the level sets of $u$, and by $\mathrm{D}$ the covariant derivative of $h$.
From the Bochner Formula, using the above expressions for the components of the hessian and the static equations~\eqref{eq:pb}, we then compute
%\begin{align*}
%|\DD u|^2
%&=\,(\D^2_{uu} u)^2\,+\,|\D u|^2\,|\hhh|^2
%\\
%&=\,\frac{(W')^2}{4W^2}\,+\,W\,|\mathring{\hhh}|^2\,+\,\frac{1}{2}\,W\,\HHH^2\,,
%\\
%&=\,\frac{(W')^2}{4W^2}\,+\,W\,|\mathring{\hhh}|^2\,+\,\frac{1}{2}\left(3 u\,-\,\frac{W'}{2 W}\right)^2\,
%\end{align*}
%We can now substitute in the Bochner Formula:
%\begin{align}
%\notag
%\De W\,&=\,2\,|\DD u|^2\,+\,2\,\Ric(\D u,\D u)\,+\,2\,\langle\D\De u\,|\,\D u\rangle
%\\
%\notag
%&=\,\frac{(W')^2}{2W^2}\,+\,2\,W\,|\mathring{\hhh}|^2\,+\,\left(3 u\,-\,\frac{W'}{2 W}\right)^2\,+\,\frac{2}{u}\,\DD u(\D u,\D u)
%\\
%\label{eq:boch}
%&=\,2\,\frac{(W')^2}{4W^2}\,+\,2\,|\mathring{\mathrm{D}}^2 u|_h^2\,+\,\left(3 u\,-\,\frac{W'}{2 W}\right)^2\,+\,\frac{1}{u}\,W\,W'\,.
%\end{align}
\begin{align}
\notag
\De W\,&=\,2\,|\DD u|^2\,+\,2\,\Ric(\D u,\D u)\,+\,2\,\langle\D\De u\,|\,\D u\rangle
\\
\notag
&=\,2\,(\D^2_{uu} u)^2\,+\,2\,|\D u|^2\,|\hhh|^2\,+\,\frac{2}{u}\,\DD u(\D u,\D u)
\\
\notag
&=\,2\,\frac{(W')^2}{4W^2}\,+\,2\,W\,|\mathring{\hhh}|^2\,+\,W\,\HHH^2\,+\,\frac{1}{u}\,\langle\D W\,|\,\D u\rangle
\\
\label{eq:boch}
&=\,\frac{(W')^2}{2W^2}\,+\,2\,W\,|\mathring{\hhh}|^2\,+\,\left(3 u\,-\,\frac{W'}{2 W}\right)^2\,+\,\frac{1}{u}\,W\,W'\,,
\end{align}
where we have denoted by $\mathring\hhh_{ij}=\hhh_{ij}-(\HHH/2)g_{ij}$ the traceless part of the second fundamental form $\hhh$.
On the other hand, notice that $\De W={\rm div}(W'\D u)=W''|\D u|^2+W'\De u=WW''+3uW'$. Plugging this information inside~\eqref{eq:boch}, we obtain
\begin{equation}
\label{eq:traceless_hess}
|\mathring{\hhh}|^2\,=\,\frac{1}{2}\,W''\,+\,\frac{3}{2}\,u\,\frac{W'}{W}\,-\,\frac{(W')^2}{4W^3}\,-\,\frac{1}{2W}\left(3 u\,-\,\frac{W'}{2 W}\right)^2\,-\,\frac{1}{2u}\,W'\,.
\end{equation}
It follows that the quantity $|\mathring{\hhh}|^2$ is a function of $u$. On the other hand, notice that the Kottler solution with nonpositive mass clearly satisfies the hypotheses of Theorem~\ref{thm:main_rewr}, hence
formula~\eqref{eq:traceless_hess} must hold in particular for the Kottler solution, in which case it is clear that $\mathring{\hhh}$ vanishes pointwise because of the warped product structure. Therefore, the function on the right hand side of~\eqref{eq:traceless_hess} must be zero for all values of $u$ (one can also show this by a direct, albeit rather cumbersome, calculation). We have deduced that $\mathring{\hhh}\equiv 0$, which implies that
$$
0\,=\,|\D u|\,\mathring{\hhh}_{ij}\,=\,\D^2_{ij} u\,-\,\frac{\De u-\D^2_{uu} u}{2}\,g_{ij}\,=\,-\,\frac{W}{2}\,\pa_u g_{ij}\,-\,\frac{6\,u\,-\,W'/W}{4}\,g_{ij}\,.
$$
In particular, we have $\pa_u g_{ij}=f(u)g_{ij}$ for some function $f$, which gives $g_{ij}=e^{F(u)}\lambda_{ij}$, where $F$ is a primitive of $f$ and the coefficients $\lambda_{ij}$ depend only on the coordinates $\vartheta^1,\vartheta^2$. It follows that the metric~\eqref{eq:almost_warp_prod} is a warped product. Since it is well known that the only conformally compactifiable warped product solutions to~\eqref{eq:pb} are the Kottler solutions (see for instance~\cite{Kobayashi} or~\cite[Section~2.2]{Bor-thesis}), this concludes the proof.
\end{proof}

\section{Further comments and generalizations}
\label{sec:future}

In this last section we are going to comment on our proof and on possible future developments and generalizations. 

\subsection{Degenerate horizons}
\label{sub:degenerate} 

In this paper, we have always assumed that the static equations~\eqref{eq:pb} extend to the boundary, in the sense that the metric $g$ is well defined on $\pa M$ and if we take the limits of both the left hand side
and the right hand side of the equations in~\eqref{eq:pb}, they coincide on $\pa M$. If one drops this hypothesis, then the surface gravity of the horizons
% is still constant on any horizon, but it 
is no longer necessarily strictly positive. We can then talk about {\em degenerate horizons}, that is, connected components of $\pa M$ with $|\D u|\equiv 0$.
%In this paper we have chosen to focus on nondegenerate horizons for simplicity. As already mentioned in Remark~\ref{rem:main}, however, our strategy works in the degenerate case as well. 
%A {\em degenerate horizon} is a connected component $S$ of $\pa M$ with $|\D u|=0$.
One can show that the vanishing of the surface gravity implies that geodesics in $M$ do not reach $\pa M$ in a finite time (see for instance~\cite[Lemma~3]{Khu_Woo}). For this reason, degenerate horizons should not be thought as boundary components, but as ends of the manifold along which the static potential $u$ goes to zero. An example of a solution with a degenerate horizon is the so called {\em critical Kottler solution}, that is, the Kottler solution~\eqref{eq:kottler} with $\kappa=-1$ and $m=\mcrit$, where we have set $\mcrit=-1/(3\sqrt{3})$.

In this subsection, we argue that Theorem~\ref{thm:main} remains true for degenerate horizons as well. Let $(M,g,u)$ be a $\mathscr{C}^3$-conformally compactifiable solution to problem~\eqref{eq:pb} whose conformal infinity $(\pa M_\infty,\hat g)$ has constant sectional curvature equal to $-1$. We want to show that, {\em if all the components of $\pa M$ are degenerate horizons and ${\rm genus}(\pa M_\infty)\leq{\rm genus}(S)$ for some horizon $S$, then $(M,g,u)$ is isometric to the critical Kottler solution.}

The proof follows exactly the same steps as the nondegenerate case. We compare the triple $(M,g,u)$ with the critical Kottler solution $(M_0,g_0,u_0)$ with mass $m_0=\mcrit$ and we define the functions $W$ and $W_0$ as before. Notice that both $W$ and $W_0$ go to zero when we approach the degenerate horizons. It follows immediately that Proposition~\ref{pro:grad_est} is in force in the degenerate case as well, meaning that $W\leq W_0$ on the whole $M$. Concerning Proposition~\ref{pro:area_bound}, since the metric does not extend to a degenerate horizon, we cannot talk about the area of $S$. It is convenient to replace the quantity $|S|$ appearing in the statement of the proposition with the limit of the area of suitable slices of a collar of $S$. 
%We consider a diffeomorphism of a collar of $S$ with $(0,1)\times S$, chosen in such a way that the slices $S_\ep=\{\ep\}\times S$ approach the horizon $S$ as $\ep\to 0$. 
In~\cite{Chr_Rea_Tod} it is shown that one can consider coordinates $(s,x^1,x^2)$ (the restriction of the Gaussian null coordinates~\cite{Mon_Ise} to our spatial slice) such that the slices $S_\ep=\{s=\ep\}$ approach the horizon $S$ as $\ep\to 0$ and the metric $g_\ep$ induced by $g$ on $S_\ep$ converges to a metric on $S$ with constant sectional curvature equal to $-3$ as $\ep\to 0$. As a consequence, we compute
$$
\lim_{\ep\to 0}|S_\ep|\,=\,\lim_{\ep\to 0}\int_{S_\ep}d\sigma_{g_\ep}\,=\,-\frac{2}{3}\,\pi\,\chi(S)\,=\,4\,\pi\,r^2_{\mcrit}\big[{\rm genus}(S)-1\big]\,.
$$
Here $r_{\mcrit}=1/\sqrt{3}$ is the radius of the degenerate horizon of the critical Kottler solution.
This proves that Proposition~\ref{pro:area_bound} holds (with equality) if we replace $|S|$ with $\lim_{\ep\to 0}|S_\ep|$.

The last modification that we need to make is in the choice of the domain $\Omega$ in Proposition~\ref{pro:mono}, as we need to remove a collar of every degenerate horizon in order to make $\Omega$ compact. We achieve this by considering Gaussian null coordinates $\{s_i,x_i^1,x_i^2\}$ on  any connected component $S_i$ of $\pa M$ and then defining the domain via
$$
\Omega\,\,=\,\,M\setminus\left(\bigcup_{i} \{s_i\leq\ep\}\cup\{\rho\geq R\}\right)\,,
$$
where $\rho$ is the coordinate on a collar of infinity used in~\eqref{eq:strong_expansions}. Proposition~\ref{pro:mono} is then proved by applying formula~\eqref{eq:nonnegative_div} to this domain $\Omega$ and then taking the limit as $\ep\to 0$ and $R\to\infty$.
The remainder of the proof works without the need of any further changes.

\subsection{Further remarks on the proof}
\label{sub:relax_hyp} 

We emphasize that the strong hypotheses of conformal compactifiability and constant sectional curvature of the conformal infinity were not so pervasively exploited in our proof. For instance, the mass aspect did not play any role. In fact, we only had to look at the asymptotics of $u$ three times:
\begin{enumerate}
\item[$(i)$] to study the asymptotic behavior of $W-W_0$, in order for Proposition~\ref{pro:grad_est} to hold,

\smallskip

\item[$(ii)$] in formula~\eqref{eq:limit_div}, in order to compute
$$
\lim_{R\to \infty}\,\,\int_{\{\rho=R\}}\bigg\langle \frac{\D u}{\Psi^3+m}\,\bigg|\,\nu\bigg\rangle\,d\sigma_{\hat g_r}\,=\,|\pa M_\infty|_{\hat g}\,,
$$

\smallskip

\item[$(iii)$] in formula~\eqref{eq:lim_theo}, namely to write
$$
|\pa M_\infty|_{\hat g}\,=\,-\int_{\pa M_\infty}\!\!\kappa_{\hat g}\,d\sigma_{\hat g}\,.
$$
\end{enumerate}
It is worth remarking that these properties independently work under weaker conditions on the asymptotics. Under the only assumption of conformal compactifiability, from expansions~\eqref{eq:strong_expansions},
%it has been shown in~\cite{Hij_Mon}, following a previous analysis in~\cite{Lee_4}, that $u$ and $g$ admit the following expansions in a collar of $\pa M_\infty$
%\begin{equation}
%\label{eq:expansions}
%u\,=\,\frac{1}{r}+\,v\,r\,,\,\qquad g\,=\,\frac{1}{r^2}\left(dr\otimes dr\,+\,\hat g_r\right)\,,
%\end{equation}
%where $\alpha>0$ and $v\in\mathscr{C}^{2,\alpha}\cap\mathscr{C}^0(\overline{M})$ satisfies
%$$
%v_{|_{\pa M_\infty}}\,=\,\frac{\kappa_{\hat g}}{4}\,,\quad |\D v|\,=\,\mathcal{O}(r^{\frac{\alpha}{2}})\,.
%$$
%Here, as usual, $r$ is the special defining function of $\pa M_\infty$ and $r\mapsto \hat g_r$ is a curve of metrics on $\pa M_\infty$, with $\hat g_0=\hat g$ being the metric induced by $u^{-2}g$ on $\pa M_\infty$.
%
%From these expansions, 
proceeding as in the proof of Proposition~\ref{pro:mono}, it is immediate to show that formula~\eqref{eq:limit_div} is in force. 
In other words, point $(ii)$ above works under the sole assumption of conformal compactifiability. Concerning point $(i)$, let us observe that, in order for Proposition~\ref{pro:grad_est} to hold, it is sufficient to show that $W-W_0$ has a nonpositive limit at infinity.
On the other hand, it is not hard to employ~\eqref{eq:strong_expansions} to compute the following estimate as $\rho\to\infty$
\begin{equation}
\label{eq:WminusWo_est}
W\,-\,W_0\,=\,-\,\kappa_{\hat g}\,-\,1\,+\,o(1)\,.
\end{equation}
It follows that, in order for Proposition~\ref{pro:grad_est} to hold, it is sufficient to have $\kappa_{\hat g}\geq -1$ pointwise. Finally, notice that points~$(i)$ and~$(ii)$ are enough to prove Proposition~\ref{pro:mono}. 
It follows that {\em Proposition~\ref{pro:mono}, and in particular the inequality
\begin{equation}
\label{eq:mono2}
|\pa M_\infty|_{\hat g}\,\geq\,4\,\pi\,\big[{\rm genus}(S)-1\big]\,,
\end{equation}
remain true under the weaker hypothesis $\kappa_{\hat g}\geq -1$.}
On the other hand, the bound $\kappa_{\hat g}\geq -1$ immediately implies
\begin{equation}
\label{eq:lim_theo2}
|\pa M_\infty|_{\hat g}\,=\,\int_{\pa M_\infty}\!\!d\sigma_{\hat g}\,\geq\,-\int_{\pa M_\infty}\!\!\kappa_{\hat g}\,d\sigma_{\hat g}
\,=\,4\,\pi\,\big[{\rm genus}(\pa M_\infty)-1\big]\,.
\end{equation}
Unfortunately, it is not possible to combine~\eqref{eq:mono2} and~\eqref{eq:lim_theo2} in order to obtain Theorem~\ref{thm:main}, unless one requires $\kappa_{\hat g}\equiv -1$ pointwise.

Concerning the assumption about the maximum surface gravity $k$ being less than or equal to $1$, this has been used only once but in a crucial point, namely in the proof of the gradient estimate $W\leq W_0$. In fact, the hypothesis $k\leq 1$ grants us that the Kottler solution $(M_0,g_0,u_0)$ we are comparing with has mass $m_0\leq 0$, which in turn implies that the coefficient $\alpha$ appearing in the elliptic inequality~\eqref{eq:elliptic} is nonpositive. The nonpositivity of the $0$-th order term of the elliptic inequality~\eqref{eq:elliptic} is needed in order to be able to apply the Maximum Principle, leading to the proof of Proposition~\ref{pro:grad_est}. If it were possible to prove that $W\leq W_0$ by other means (for instance if it were possible to find a different elliptic inequality whose $0$-th order term is nonpositive without assumptions on the sign of $m_0$), then Theorem~\ref{thm:main} would work without the hypothesis on the bound on the surface gravity.

\subsection{Other generalizations}
\label{sub:generalizations} 

It is natural to ask if our proof can be adapted to study Kottler solutions whose slices have nonnegative sectional curvature. In particular, it would certainly be interesting to find a uniqueness result in the spirit of Theorem~\ref{thm:main} for the Kottler solution~\eqref{eq:kottler} with $\kappa=1$, usually referred to as the Schwarzschild--Anti de Sitter solution. The crucial problem seems to be again that of proving that the estimate $W\leq W_0$ is in force. As already mentioned at the end of the previous subsection, in order for the $0$-th order term of the elliptic inequality~\eqref{eq:elliptic} to have the right sign, one needs nonpositive masses, whereas for Kottler solutions with $\kappa\geq 0$, only parameters $m>0$ are allowed.

We conclude by briefly addressing the higher dimensional case. It is not hard to adapt our arguments to study static solutions $(M^n,g,u)$ of any dimension $n\geq 3$. At the cost of more cumbersome computations, one can still prove that $W-W_0$ satisfies an elliptic inequality, leading to the gradient estimate $W\leq W_0$ when $m_0\leq 0$. We do not give the details, but the interested reader can essentially retrace the steps in~\cite{Bor_Maz_2-II}: in that paper, the case of a positive cosmological constant in any dimension is studied, and a gradient estimate analogous to Proposition~\ref{pro:grad_est} is proven (\cite[Proposition~3.3]{Bor_Maz_2-II}). Once the gradient estimate is achieved, it is easy to adjust the remaining arguments in Proposition~\ref{pro:mono} to show the inequality
\begin{equation}
\label{eq:area_bound_n}
r_m^{n-1}|\pa M_\infty|_{\hat g}\,\geq\,|S|\,.
\end{equation}
In dimension $n=3$, one can exploit the Gauss-Bonnet formula to relate the areas and genuses of $S$ and $\pa M_\infty$ (this has been done in Proposition~\ref{pro:area_bound} and formula~\eqref{eq:lim_theo}). Unfortunately, this is not possible when $n>3$, and we are left without a clear way to improve on~\eqref{eq:area_bound_n}.

% BIBLIOGRAFIA CON BIBTEX %
\bibliographystyle{plain}
\bibliography{biblio}

\end{document}